\documentclass[11pt,reqno]{article}
\usepackage{amssymb,amscd,amsmath,amsthm}

\usepackage{graphicx}
\newtheorem{theorem}{Theorem}[section]
\newtheorem{proposition}[theorem]{Proposition}
\newtheorem{lemma}[theorem]{Lemma}
\newtheorem{corollary}[theorem]{Corollary}
\newtheorem{definition}[theorem]{Definition}

\def\bloc{{\cal B}_{L,{\rm loc}}}

\def\cb{{\mathcal B}}

\def\ce{{\mathcal E}}

\def\cw{{\mathcal W}}

\def\bc{{\mathbb C}}

\def\bn{{\mathbb N}}

\def\br{{\mathbb R}}

\def\a{\alpha}
\def\b{\beta}

\def\tr{{\rm Tr}}
\def\L{\Lambda}
\def\G{\Gamma}
\def\ce{\mathcal E}

\def\ffi{\varphi}

\def\<{\langle}
\def\>{\rangle}

\def\1{\mathbf{1}}

\def\kxy{K_{<x,y>}}
\def\kxz{K_{<x,z>}}

\def\cw{\cal W}

\def\cal{\mathcal}

\def\ssub{\subset \subset}
\def\subf{\subseteq_{\rm fin}}
\def\subfc{\subseteq_{\rm fin,c}}

\def\si{\sinh 2\beta}

\def\dis{{\rm dist}}
\def\id{{\bf 1}\!\!{\rm I}}
\begin{document}


\centerline{\Large {\bf Quantum Markov fields on graphs }}

\begin{center}
Luigi Accardi, \\
Centro Vito Volterra,\\
Universit\`a di Roma ``Tor Vergata",\\
Roma I-00133, Italy\\
E-mail: {\tt accardi@volterra.uniroma2.it}
\end{center}
\begin{center}
Hiromichi Ohno, \\
Faculty of Engineering, Shinshu University\\
4-17-1 Wakasato, Nagano 380-8553, Japan\\
E-mail: {\tt  h\_ohno@shinshu-u.ac.jp}
\end{center}

\begin{center}
Farrukh Mukhamedov,\\
Department of Computational \& Theoretical Sciences,\\
Faculty of Science, International Islamic University Malaysia,\\
P.O. Box, 141, 25710, Kuantan, Pahang, Malaysia\\
E-mail: {\tt far75m@yandex.ru, \ farrukh\_m@iiu.edu.my}
\end{center}

\begin{abstract}
We introduce generalized quantum Markov states and
generalized d-Markov chains which extend the
notion quantum Markov chains on spin systems to that on
$C^*$-algebras defined by general graphs. As examples of
generalized d-Markov chains,
 we construct the entangled Markov fields on
tree graphs.
The concrete examples of generalized d-Markov chains on Cayley trees
are also investigated.

\end{abstract}

\section{  Introduction }\label{intr}

 Markov fields play an important role in classical probability, in physics, in
biological and neurological models and in an increasing number of
technological problems such as image recognition.

It is quite natural to forecast that the quantum analogue of these
models will also play a relevant role.

The papers \cite{[Liebs99]}, \cite{[AcFi01a]},
\cite{[AcFi01b]},\cite{AcLi} are a first attempts to construct a
quantum analogue of classical Markov fields. These papers extend to
fields the notion of {\it quantum Markov state} introduced in
\cite{[AcFr80]} as a sub--class of the {\it quantum Markov chains}
introduced in \cite{[Ac74f]}. As remarked in \cite{[Liebs99]}, the
peculiarity of the former class of states with respect to the latter
consists in the fact that they admit a Umegaki conditional
expectation {\it into} rather than {\it onto} their range.

This small difference allows, when applied to states on infinite
tensor products of $C^*$--algebras, to obtain nontrivial (i.e. non
product) states while maintaining most of the simple algebraic
properties related to classical Markovianity.

The prize one has to pay for this simplification is that the
resulting class of states, although non trivial, has very poor
entanglement properties so that they cannot exhibit some of the
most interesting properties which distinguish the quantum from the
classical world.

On the contrary the quantum Markov chains or, more generally, the
{\it generalized quantum Markov states} in the sense of
\cite{[Ohno03]} may exhibit very strong entanglement properties.
In particular the paper \cite{[Miyad05]} shows that this is indeed
the case for the {\it entangled Markov chains} constructed in
\cite{[AcFi03]}.
A degree of entanglement of entangled Markov chains is
 considered in \cite{[AcMaOh]}.

The above considerations naturally suggest the study of following
two problems:
\begin{enumerate}

\item[(i)]  the extension to fields of the notion of generalized
Markov state (or Markov chain)

\item[(ii)] the extension to fields of the construction of
entangled Markov chains produced in \cite{[AcFi03]}

\end{enumerate}

The present paper is a first step towards the solution of these
problems. We introduce a hierarchy of notions of Markovianity for
states on discrete infinite tensor products of $C^*$--algebras
(Section \ref{dfqmf}) and for each of these notions we construct
some explicit examples. We show that the construction of
\cite{[AcFi03]} can be generalized to trees (Section \ref{tree}).
 It is interesting to notice that, in a different
context and for quite different purposes, the special role of trees
was already emphasized in \cite{[Liebs99]}. Note that in
\cite{fannes} finitely correlated states are constructed as ground
states of VBS-model on Cayley tree. As well as, such shift invariant
$d$-Markov chains can be  also considered as an extension of
$C^*$-finitely correlated states defined in \cite{fannes2} to the
Cayley trees. In the classical case, Markov fields on trees are also
considered in \cite{[Pr]}-\cite{[Za85]}.

A comment on the  notion of generalized quantum Markov state
introduced in Definition \ref{dfgenwkMS} may help understanding
the logic leading to this definition and in particular condition
(\ref{dfgqmf}) which otherwise might, at first sight, seem
artificial.

The point is that, as we know from Dobrushin's seminal work
\cite{[Dobr68a]}, the natural localization for fields on a
discrete set $L$ is given by the finite subsets of $L$ and their
complements. This localization, when restricted to the
$1$--dimensional case, does not lead to the usual probabilistic
localization but, in a certain sense to its {\it dual} (or {\it
time reversal}), corresponding to the conditioning of the past on
the future rather than conversely.  This leads to different
structures of the Markov chains in the two cases, a fact already
noted in \cite{[Ac74f]} where these two types were called {\it
Markov chains} and {\it inverse Markov chains} respectively.

In particular the role played by the time zero algebra in the
usual Markov processes is played by the algebra at infinity in the
multi--dimensional case.

But, while the time zero algebra has a meaning independent of the
state, the algebra at infinity can be (meaningfully) defined only
in the GNS representation of the given state. Therefore, if one
wants to give a constructive and local definition of a state one
cannot make use of a global notion such as the algebra at
infinity.

In the ergodic cases, corresponding physically to the {\it pure
phases} in Dobrushin's theory, one expects that the algebra at
infinity is trivial and that the sequence of conditional
expectations appearing in (\ref{dfgqmf}) converges weakly to a
single state (asymptotic independence of the boundary) so that the
resulting state is in fact independent of the sequence of states
$(\hat\varphi_{\Lambda_n^c})$ which plays the role of the single
``state" $\hat\varphi_{L^c} = \hat\varphi_{\infty}$, not available
at a $C^*$--level.

Let us briefly mention about the organization of the paper. In
Sections \ref{graphs} and \ref{bundles}, we introduce definition
of graphs and bundles of graphs, and in Section \ref{dfqmf} {\it
generalized quantum Markov states} and {\it d-Markov chains} on
graphs are defined. In the further Sections \ref{tree} we
provide examples of generalized quantum Markov chains which extend
the entangled Markov chains, defined in \cite{[AcFi03]}, to tree
graphs and general graphs. In Section \ref{dfcayley}, we consider
a particular case of tree, so called Cayley tree. Over such a tree
we give a construction of $d$-Markov chains, in next sections
\ref{exam1} and \ref{exam2} we provide some more concrete examples
of such chains, which are shift invariant and have the clustering
property.


\section{Graphs}\label{graphs}

Let $\mathcal{G}=(L,E)$ be a (non-oriented simple) graph, that is,
$L$ is a non-empty at most countable set and
$$E\subset\{\{x,y\}:x,y\in L, x\neq y\}.$$
Elements of $L$ and of $E$ are called \textit{vertices} and
\textit{edges}, respectively. Two vertices $x,y\in L$ are called
\textit{adjacent}, or {\it nearest neighbors}, if $\{x,y\}\in E$,
and in that case we also write $x\sim y$.

For each $x\in L$, the set of nearest neighbors of $x$ will be
denoted by
$$
N(x) :=\{y\in L:y\sim x\}.
$$
The \textit{degree} of $x\in L$, denoted by $\kappa(x)$, is the
number of vertices adjacent to $x$, namely,
$$
\kappa(x) := |N(x)| = |\{y\in L\,:\, y\sim x\}|,
$$
where $|\cdot|$ is the cardinality.




A graph can be equivalently assigned by giving the pair
$$(L,\sim)$$
of its vertices and the binary symmetric relation $\sim $.

A {\it path\/} or a {\it trajectory\/} or a \textit{walk}
connecting two points $x,y\in L$ is a finite sequence of vertices
such that $x=x_1\sim x_2\sim\dots\sim x_{n} =y$. In this case
$n-1$ is called the \textit{length} of the walk. For two distinct
vertices $x,y\in L$, the distance $\dis (x,y)$ is defined to be the
shortest length of a walk connecting $x$ and $y$. By definition
$\dis (x,x)=0$.

Throughout the paper we always assume that a graph is locally
finite, i.e., $\kappa(x)<\infty$ for all $x\in L$, and is
connected, i.e., for any pair of vertices, there exists a walk
connecting them. We will write
$$
\Lambda\subseteq_{\text{fin}}L, \qquad \Lambda\subseteq_{\text{fin,c}}L
$$
to mean that $\Lambda$ is a finite subset
and a finite connected subset of $L$, respectively.
Given $\Lambda\subseteq_{\text{fin}}L$ we define the
external boundary of $\Lambda$ by
$$\vec\partial\Lambda:=\{x\in\Lambda^c:
\,y\sim x\ ,\,\, \exists y\in\Lambda\}$$ and the closure of
$\Lambda$ by
$$
\overline \Lambda := \Lambda \cup \vec\partial\Lambda.
$$
We will write
$$
\Lambda\subset \subset \Lambda_1
$$
to mean that $ \overline \Lambda \subset \Lambda_1. $ Notice that,
by definition
$$
\Lambda\cap\vec\partial\Lambda=\emptyset,
$$
$$\vec\partial\{ x \} =: \vec\partial x = N(x).$$


\section{Bundles on graphs}\label{bundles}

\setcounter{equation}{0} To each $x\in L$ it is associated an
Hilbert space ${\mathcal H}_x$ of dimension $d_{\mathcal
H}(x)\in\Bbb N$. In the present paper we will assume that
$$
d:= d_{\mathcal H}(x) = d_{\mathcal H} < + \infty \qquad
(\hbox{independent of } \ x).
$$
Given $\Lambda\subseteq_{\text{fin}}L$ we define
$${\mathcal H}_\Lambda:=\bigotimes_{x\in \Lambda}{\mathcal H}_x.$$
For each $x$ in $L$, we fix an orthonormal basis of ${\mathcal
H}_x$:
$$
\{ e_j(x)\} \quad;\quad j\in S(x):=\{1,\dots,d\}.
$$
When we consider $S$ as a total space, $\pi_S: S \to L$ is the
bundle whose fibers are the finite sets $\pi^{-1}_S(x):=S(x)$ and
the sections of this bundle are the maps:
$$
{\cal F}(\Lambda,S):=\{\omega_\Lambda :x\in\Lambda\mapsto \omega
_\Lambda (x) \in S(x)\}.
$$
A section $\omega_\Lambda $ is also called a {\it configuration}
in the volume $\Lambda $. For $\omega_\Lambda\in{\cal
F}(\Lambda,S)$, the vector $e_{\omega_\Lambda }$ is defined by
\begin{equation}\label{vecoml}
e_{\omega_\Lambda }:=\bigotimes_{x\in \Lambda } e_{\omega_\Lambda
(x)}(x)\in{\cal H}_\Lambda
\end{equation}
and we will use the symbol $P_{\omega_\Lambda}$ for the
corresponding rank one projection:
\begin{equation}\label{prjoml}
P_{\omega_\Lambda}:=|e_{\omega_\Lambda}\rangle\langle
e_{\omega_\Lambda}| =e_{\omega_\Lambda} e_{\omega_\Lambda}^*.
\end{equation}
Then the set
\begin{equation}
\{e_{\omega_\Lambda} \ : \ \omega_\Lambda\in{\cal
F}(\Lambda,S)\}\label{2.3}
\end{equation}
is an orthonormal basis of ${\cal H}_\Lambda$. Thus the generic
vector of ${\cal H}_\Lambda$ has the form
$$
\sum_{\omega_\Lambda\in{\cal F}(\Lambda,S)}
\lambda_{\omega_\Lambda}e_{\omega_\Lambda}.
$$
We will use the notation
$${\cal B}_\Lambda:=
{\cal B}({\cal H}_\Lambda)$$ for each
$\Lambda\subseteq_{\text{fin}}L$ and ${\cal B}_L$ is the inductive
$C^*$-algebra, that is,
$$
{\cal B}_L:=\lim_{\longrightarrow} {\cal B}_\Lambda
$$
for $\Lambda \uparrow L$. As a $C^*$-algebra ${\cal B}_L$ is
isomorphic to the (unique) infinite $C^*$-tensor product
$\bigotimes_{x\in L}{\cal B}_x$, the natural embedding of ${\cal
B}_x $ into ${\cal B}_L$ will be denoted by
\begin{equation}
j_x:b\in{\cal B}_x\mapsto j_x(b)=b\otimes I_{\{x\}^c}\in{\cal
B}_L.\label{natemb}
\end{equation}
Similarly, for $\Lambda\subf L$, we define
$$j_\Lambda:=\bigotimes_{x\in\Lambda}j_x.$$
To simplify the notations, in the following we will often identify
each ${\cal B}_\Lambda$ to the subalgebra $j_\Lambda({\cal
B}_\Lambda)$ of ${\cal B}_L$, through the identification
$${\cal B}_\Lambda\equiv{\cal B}_\Lambda\otimes
I_{\Lambda^c}=j_\Lambda({\cal B}_\Lambda).$$ With these notations
the elements of the $*$-subalgebra of ${\cal B}_L$ defined by
$$
{\cal B}_{L,{loc}}:=\bigcup_{\Lambda \subseteq_{\text{fin}}L}{\cal
B}_\Lambda
$$
will be called a local algebra or local operators (observables if
self--adjoint).

In what follows, by ${\cal S}({\cal B}_\L)$ we will denote the set
of all states defined on the algebra ${\cal B}_\L$.


\section{Definition of generalized quantum Markov state}\label{dfqmf}

Consider a triplet ${\cal C} \subset {\cal B} \subset {\cal A}$ of
unital $C^*$-algebras. Recall that a {\it quasi-conditional
expectation} with respect to the given triplet is a completely
positive identity preserving linear (CP1) map $\ce \,:\, {\cal A}
\to {\cal B}$ such that \begin{equation}\label{qe} \ce(ca) = c
\ce(a), \qquad a\in {\cal A},\, c \in {\cal C}.
\end{equation}
 Notice that, as
the quasi-conditional expectation $\ce$ is a real map, one has
\[
\ce(ac) = \ce(a)c, \qquad a\in {\cal A},\, c \in {\cal C}.
\]
as well.

\begin{definition}\label{dfgenwkMS}
A state $\varphi$ on ${\cal B}_L$ is called a generalized quantum
Markov state  on ${\cal B}_L$ if there exist an increasing
sequence of finite sets $\Lambda_n \uparrow L$ with
$\Lambda_n \ssub \Lambda_{n+1}$ and, for each
$\Lambda_n$, a quasi-conditional expectation $\ce_{\Lambda_n^c}$
with respect to the triplet
\begin{equation}\label{trplt1}
{\cal B}_{\overline{\Lambda}_n^c}\subseteq {\cal
B}_{\Lambda_n^c}\subseteq{\cal B}_{\Lambda_{n-1}^c}
\end{equation}
and a state
$$
\hat\varphi_{\Lambda_n^c}\in{\cal S}({\cal B}_{\Lambda_n^c})
$$
such that for any $n\in {\mathbb N}$ one has
\begin{equation}\label{eq4.1re}
\hat\varphi_{\Lambda_n^c}| {\cal B}_{\Lambda_{n+1}\backslash\Lambda_n} =
\hat\varphi_{\Lambda_{n+1}^c}\circ
\ce_{\Lambda_{n+1}^c}| {\cal B}_{\Lambda_{n+1}\backslash\Lambda_n}
\end{equation}
and
\begin{equation}\label{dfgqmf}
\varphi = \lim_{n\to\infty} \hat\varphi_{\Lambda_n^c}\circ
\ce_{\Lambda_n^c}\circ \ce_{\Lambda_{n-1}^c} \circ
\cdots \circ \ce_{\Lambda_1^c}
\end{equation}
in the weak-* topology.

\end{definition}

In this definition, a generalized quantum Markov state $\varphi$
generated by ${\cal E}_{\Lambda_n^c}$ and
$\varphi_{\Lambda_n^c}$ is well-defined. Indeed, we have
\[
\hat\varphi_{\Lambda_n^c} \circ
\ce_{\Lambda_{n}^c} |\cb_{\Lambda_n}=
\hat\varphi_{\Lambda_{n+1}^c}\circ
\ce_{\Lambda_{n+1}^c} \circ \ce_{\Lambda_n^c}| {\cal B}_{\Lambda_n}
\]
by (\ref{eq4.1re}) and a following remark so that,
for $\Lambda \ssub \Lambda_k$ and $a\in {\cal B}_{\Lambda}$,
\[
\lim_{n\to\infty}  \hat\varphi_{\Lambda_n^c}\circ
\ce_{\Lambda_n^c}\circ \ce_{\Lambda_{n-1}^c} \circ
\cdots \circ \ce_{\Lambda_1^c}(a) =
\hat\varphi_{\Lambda_k^c}\circ
\ce_{\Lambda_k^c}\circ \ce_{\Lambda_{k-1}^c} \circ
\cdots \circ \ce_{\Lambda_1^c}(a).
\]

\noindent{\bf Remark}.
Markov states on multi-dimensional lattice ${\mathbb Z}^\nu$
introduced in \cite{[AcFi01a]} are generalized quantum Markov states.
Indeed, define an increasing
sequence of finite sets $\Lambda_n \uparrow L$.
Then for any $\Lambda_n$, there is a conditional expectation
$\ce_{\Lambda_n^c}$ from $\cb_{L}$ to $\cb_{\Lambda_n^c}$ with
$\ce_{\Lambda_n^c}(\cb_{L}) \subset \cb_{\overline{\Lambda}_n^c}$ and
\[
\varphi \circ \ce_{\Lambda_n^c} = \varphi.
\]
Let $\hat\varphi_{\Lambda_n^c} = \varphi | \cb_{\Lambda_n^c}$.
Then the Markov state $\varphi$ is a generalized quantum Markov state
generated by ${\cal E}_{\Lambda_n^c}$ and
$\varphi_{\Lambda_n^c}$.

\noindent{\bf Remark}. In the case of infinite tensor products
(the only one considered here) one has, for any subset,
$I\subseteq L$:
\begin{equation}\label{comls}
{\cal B}_{I^c} = {\cal B}'_I \quad \hbox{the commutant of } \
{\cal B}_{I}.
\end{equation}
From (\ref{qe}) for the quasi--conditional expectation
$\ce_{\Lambda_n^c}:{\cal B}_L\to{\cal B}_{\Lambda_n^c}$ with
respect to the triplet (\ref{trplt1}) one has
\begin{equation}\label{qce-bp}
\ce_{\Lambda_n^c}(a_{\overline{\Lambda}_n^c}a_{\Lambda_n})=
a_{\overline{\Lambda}_n^c} \ce_{\Lambda_n^c}(a_{\Lambda_n}).
\end{equation}
Because of (\ref{comls}) the last equality implies that
$\ce_{\Lambda_n^c}({\cal B}_{\Lambda_n})\subseteq ({\cal
B}_{\overline\Lambda_n^c})'={\cal
B}_{(\overline\Lambda_n^c)^c}={\cal B}_{\overline\Lambda_n}$.

Consequently,
$$
\ce_{\Lambda_n^c}({\cal B}_{\Lambda_n})\subseteq {\cal
B}_{\Lambda_n^c}\cap{\cal B}_{\overline{\Lambda}_n}={\cal
B}_{\vec\partial\Lambda_n}
$$
which is the natural quantum generalization of the
multidimensional (discrete) Markov property as originally
formulated by Dobrushin \cite{[Dobr68a]}.

The above argument shows that, whenever (\ref{comls}) holds (e.g.
in the case of infinite tensor products) the Markov property
$$
\ce_{\Lambda_n^c}({\cal B}_{\Lambda_n})\subseteq{\cal
B}_{\vec\partial\Lambda}
$$
follows from the basic property (\ref{qce-bp}) of the
quasi--conditional expectations. This is not true in general when
(\ref{comls}) does not hold (e.g. in the abelian case or in the
case of CAR algebras, see \cite{[AcFiMu07]}). In all these cases
the Markov property should be included in the definition of the
various notions of Markov states as an additional requirement
\cite{[AcFiMu07]}.

Next, we introduce the definition of $d$-Markov chains extending
the definition in
\cite{[Ac74f]} to the graph case. Assume $\{
\Lambda_n \}_{n=1}^\infty$ is an increasing sequence of finite
sets of $L$ such that $\overline{\Lambda}_n = \Lambda_{n+1}$ then
$\Lambda_n \uparrow L$.

\begin{definition}
A state $\ffi$ on ${\cal B}_L$ is called a $d$-Markov chain
associated to $\{ \Lambda_n \}$ if there exist a quasi-conditional
expectation $\ce_n$ with respect to the triple ${\cal
B}_{\Lambda_{n-1}} \subset {\cal B}_{\Lambda_n} \subset {\cal
B}_{\Lambda_{n+1}}$ for each $n\in {\mathbb N}$ and an initial
state $\rho$ on ${\cal B}_{\Lambda_1}$ such that
\[
\ffi = \lim_{n\to \infty} \rho \circ \ce_1 \circ \ce_2 \circ
\cdots \circ \ce_n
\]
in the weak-$*$ topology.
\end{definition}

In this definition, the state $\ffi$ is well-defined.
Indeed, since $\ce_k(a) =a$ for any $a \in \cb_{\Lambda_n}$
and $k \ge n+1$, we have
\[
\lim_{k\to \infty} \rho \circ \ce_1 \circ \ce_2 \circ
\cdots \circ \ce_k (a)
=
 \rho \circ \ce_1 \circ \ce_2 \circ
\cdots \circ \ce_n (a).
\]


\section{ Entangled Markov fields on trees }\label{tree}

In this section we prove that, for a very special class of graphs,
i.e. the trees, the construction of entangled Markov chains
proposed in \cite{[AcFi03]} can be generalized. The simplification
coming from considering trees rather than general graphs manifests
itself in the fact that the analogue of the basic isometries, used
in the construction of \cite{[AcFi03]}, in this case commute.

Recall that a tree is a connected graph without loops. This
definition implies that any finite connected subset $
\Lambda\subseteq_{\text{fin,c}} L$ enjoys the following fundamental
property:

\bigskip
\noindent
{\bf Tree Property}

{\it For any $ \Lambda\subseteq_{\rm{fin,c}} L$ and for arbitrary $x
\in \vec\partial \Lambda$, there exists a unique point $y \in
\Lambda$ such that $x\sim y$}.
\bigskip

${\mathbb N}$, ${\mathbb Z}$ and Cayley trees are examples of tree graphs
and general tree graphs have a form as in Fig 1.

\begin{center}
\unitlength 0.1in
\begin{picture}( 28.0000, 12.1000)( 22.0000,-30.1000)
%
\special{pn 13}%
\special{pa 3600 3010}%
\special{pa 3200 2400}%
\special{fp}%
%
\special{pn 13}%
\special{pa 3600 3000}%
\special{pa 4000 2400}%
\special{fp}%
\special{pa 4000 2400}%
\special{pa 4000 2400}%
\special{fp}%
%
\special{pn 13}%
\special{pa 2400 2400}%
\special{pa 3600 3000}%
\special{fp}%
\special{pa 4810 2400}%
\special{pa 3600 3000}%
\special{fp}%
\special{pa 4000 2400}%
\special{pa 4010 1800}%
\special{fp}%
\special{pa 3200 2400}%
\special{pa 3400 1800}%
\special{fp}%
\special{pa 3000 1800}%
\special{pa 3200 2410}%
\special{fp}%
%
\special{pn 13}%
\special{pa 4800 1800}%
\special{pa 4800 2400}%
\special{fp}%
\special{pa 4600 1800}%
\special{pa 4800 2400}%
\special{fp}%
\special{pa 5000 1800}%
\special{pa 4800 2400}%
\special{fp}%
\special{pa 2400 2400}%
\special{pa 2600 1800}%
\special{fp}%
\special{pa 2200 1800}%
\special{pa 2400 2400}%
\special{fp}%
\end{picture}%

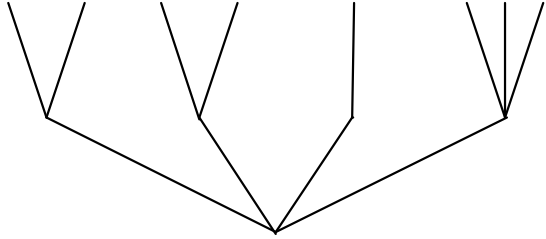
Fig 1: example of tree graphs
\end{center}

The fact that Tree Property is the main ingredient used in the
proofs of the results below justifies the expectation that our
results could be generalized to any graph such that there exists a
sequence of $ \Lambda_n\subseteq_{\text{fin,c}} L$ such that
 $ \Lambda_n\uparrow L$ and each  $ \Lambda_n $ enjoys Tree Property
(maybe with the exception of a {\it small} set of points).

The trouble with Tree Property is that, if $\Lambda$ has Tree
Property and $x \in \vec\partial \Lambda$, unfortunately it is not
true that also $\Lambda\cup x $ has Tree Property. However trees
have a very special property given by the following Lemma.

\begin{lemma}\label{basproptree}
In a tree every finite connected subset $\Lambda\subseteq_{\rm{fin,c}} L $
enjoys Tree Property.
\end{lemma}

\begin{proof}
Let $\Lambda\subseteq_{\text{fin,c}} L $ be a finite connected subset and let
$x \in \vec\partial \Lambda$. If there exist $y,z\in \Lambda$ such
that $y\sim x \ , \ z\sim x$, then since a tree is connected,
there is a path between $y$ and $z$ and this would give a loop.
Against the definition of tree. \end{proof}

We keep the notations and assumptions of Section \ref{graphs}.
Let $(L,E )$ be a graph and let, for each $\{x,y\} \in E$, be
given a complex $d\times d$ matrix $(\psi_{xy} (i,j))$ such that
the matrix $(| \psi_{xy}(i,j) |^2 )$ is bi--stochastic, i.e.
$$
\sum^d_{i=1} | \psi_{xy}(i,j) |^2 =\sum^d_{j=1} | \psi_{xy}(i,j)
|^2 =1.
$$
$(\psi_{xy} (i,j))$ will be called an {\it amplitude matrix}:
notice that unitarity of the matrices $(\psi_{xy}(i,j))_{i,j}$ is
not required. Define the vector
\begin{equation}\label{dfpsixy}
\psi_{xy} = \sum_{i,j=1}^d \psi_{xy} (i,j) \cdot e_i (x)\otimes
e_j(y) \in {\cal H}_x \otimes {\cal H}_y.
\end{equation}
Moreover, in the notation
$$
E_\Lambda := \{ \{x,y\} \, |\, x,y \in \Lambda , x\sim y \}
$$
for any $\Lambda \subseteq_{\rm fin} L$,
define the vector $\psi_\Lambda \in {\cal H}_\Lambda$ by
\begin{equation}\label{psilamda}
\psi_\Lambda := \sum_{\omega_\Lambda} \psi_\Lambda
(\omega_\Lambda) e_{\omega_\Lambda},
\end{equation}
\begin{equation}\label{psilom}
\psi_\Lambda ( \omega_\Lambda ) := \prod_{\{x,y\} \in E_\Lambda}
\psi_{xy} ( \omega_\Lambda (x), \omega_\Lambda(y)).
\end{equation}

\begin{lemma}\label{conspropT}
If $\Lambda \subfc L$ enjoys Tree Property then for all $ x \in \vec\partial
\Lambda$,
$$
\| \psi_{\Lambda \cup \{x\}}\|^2 = \| \psi_{\Lambda}\|^2.
$$
\end{lemma}

\begin{proof} Tree Property implies that, for
arbitrary $x \in \vec\partial \Lambda$, there exists a unique
point $y \in \Lambda$ such that $x\sim y$. Then
\begin{eqnarray*}
\| \psi_{\Lambda \cup \{x\}} \|^2 &=& \sum_{\omega_\Lambda ,
\omega_x} | \psi_{\Lambda \cup \{x\}} ((\omega_\Lambda ,\omega_x))
|^2 =\sum_{\omega_{\Lambda \backslash \{y\}} ,\omega_y ,\omega_x}
| \psi_{\Lambda \cup \{x \}}
((\omega_{\Lambda \backslash \{y\}} ,\omega_y ,\omega_x)) |^2 \\
&=& \sum_{\omega_{\Lambda \backslash \{y\}, \omega_y }}
\sum_{\omega_x=1}^d | \psi_\Lambda ((\omega_{\Lambda \backslash
\{y\}} ,\omega_y ))|^2 \cdot
|\psi_{xy}(\omega_y,\omega_x)|^2 \\
&=& \sum_{\omega_\Lambda} |\psi_\Lambda (\omega_\Lambda) |^2 = \|
\psi_\Lambda \|^2
\end{eqnarray*}
which proves the assertion. \end{proof}

\begin{proposition}\label{treewelldefined}
Suppose that $\Lambda$ enjoys Tree Property and let
$$
\Lambda' \subset \subset \Lambda \subseteq_{\rm fin,c} L.
$$
Then for any $a \in {\cal B}_{\Lambda'}$ and $x \in \vec\partial
\Lambda$ one has:
$$
\langle \psi_{\Lambda} , a \psi_{\Lambda} \rangle = \langle
\psi_{\Lambda \cup \{x\}},a \psi_{\Lambda \cup \{x\}} \rangle.
$$
\end{proposition}

\begin{proof} Because of Tree Property, given $x \in \vec\partial
\Lambda$, there exists a unique point $y \in \Lambda$ such that $x
\sim y$. Then we have
\begin{eqnarray*}
&& \langle \psi_{\Lambda \cup \{x\}},
a\psi_{\Lambda \cup \{x\}} \rangle =\\
&=& \sum_{\omega_{\Lambda'} ,\omega_{\Lambda'}'}
\sum_{\omega_{\Lambda \backslash \{\Lambda' \cup \{y\}\}}}
\sum_{\omega_x,\omega_y} \psi_{\Lambda \cup
\{x\}}((\omega_{\Lambda'},
\omega_{\Lambda \backslash \{\Lambda' \cup \{y\}\}},\omega_x,\omega_y))^* \\
&& \cdot a_{\omega_{\Lambda'} \omega_{\Lambda'}'} \psi_{\Lambda
\cup \{x\}}((\omega_{\Lambda'}',
\omega_{\Lambda \backslash \{\Lambda' \cup \{y\}\}},\omega_x,\omega_y)) \\
&=& \sum_{\omega_{\Lambda'} ,\omega_{\Lambda'}'}
\sum_{\omega_{\Lambda \backslash \{ \Lambda' \cup \{y\} \}}}
\sum_{\omega_x,\omega_y} \psi_\Lambda((\omega_{\Lambda'},
\omega_{\Lambda \backslash \{ \Lambda' \cup \{y\} \}},\omega_y))^* \\
&& a_{\omega_{\Lambda'} \omega_{\Lambda'}'}
\psi_\Lambda((\omega_{\Lambda'}', \omega_{\Lambda \backslash \{
\Lambda' \cup \{y\}\}},\omega_y))
|\psi_{xy}(\omega_x,\omega_y)|^2 \\
&=& \sum_{\omega_{\Lambda'} ,\omega_{\Lambda'}'}
\sum_{\omega_{\Lambda \backslash \Lambda'}}
\psi_\Lambda((\omega_{\Lambda'},\omega_{\Lambda \backslash
\Lambda'}))^* a_{\omega_{\Lambda'} \omega_{\Lambda'}'}
\psi_\Lambda((\omega_{\Lambda'}',\omega_{\Lambda \backslash \Lambda'})) \\
&=& \langle \psi_\Lambda , a \psi_\Lambda \rangle,
\end{eqnarray*}
where $a_{\omega_{\Lambda'} \omega_{\Lambda'}'}
= \langle e_{\omega_{\Lambda'}}, a e_{\omega_{\Lambda'}'} \rangle$.

\end{proof}

\begin{corollary}\label{existence}
If $(L ,E)$ is a tree, and the vector $\psi_\Lambda$ is defined by
(\ref{psilamda}), (\ref{psilom}), then, for any $\Lambda
\subseteq_{\rm fin,c} L$ of cardinality $\geq 2$, one has:
\begin{equation}\label{nrmpsld}
\| \psi_\Lambda \|^2 =d
\end{equation}
and the limit
$$
\varphi(a) = \frac{1}{d} \lim_{\Lambda \uparrow L}\langle
\psi_\Lambda ,a \psi_\Lambda \rangle
$$
exists for any $a$ in the local algebra ${\cal B}_{L,loc}$ and
defines a state $\varphi$ on ${\cal B}_L$.
\end{corollary}

\begin{proof}
The first statement follows by induction from Proposition
\ref{conspropT} and Lemma \ref{basproptree} because, if $\Lambda =
\{x ,y\}$, then we get
\begin{eqnarray*}
\| \psi_{xy} \|^2 = \sum_{i,j} | \psi_{xy}(i,j) |^2 = d.
\end{eqnarray*}
The second statement follows from the first one and Proposition
\ref{treewelldefined}. \end{proof}

The obtained state in Corollary \ref{existence} is called {\it
entangled Markov filed} on ${\cal B}_L$. When $L=\mathbb{Z}$ such
a state was introduced and studied in \cite{[AcFi03], [Miyad05]}.
We will see that the state $\varphi$ is a $d$-Markov
chain and, in special case, it is a generalized quantum Markov state.

For $\Lambda \subseteq_{\rm fin,c} L$, $x \in \vec\partial \Lambda$
and $z \in \Lambda$, with $z\sim x$, define
$V_{(z|x)}: {\cal H}_z \to {\cal H}_z \otimes {\cal H}_x$ by
\begin{equation}\label{dfvzx}
V_{(z|x)} e_{i_z}= \sum_{i_x} \psi_{xz}(i_x,i_z) e_{i_x} \otimes
e_{i_z}.
\end{equation}
Then $V_{(z|x)}$ is naturally extended to an operator from
${\cal H} \otimes {\cal H}_x$ to
${\cal H} \otimes {\cal H}_x\otimes {\cal H}_z$ for any Hilbert space
${\cal H}$ by $I_{\cal H} \otimes V_{(z|x)}$.
We will also write $V_{(z|x)}$ for $I_{\cal H} \otimes V_{(z|x)}$.

\begin{proposition}\label{commute}
For any $\Lambda \subseteq_{\rm fin,c} L$, $ x ,y \in \vec\partial \Lambda$
and $z \in \Lambda$ with
$x \sim z$, $y \sim z$, $V_{(z|x)}$ and $V_{(z|y)}$ are isometries
satisfying:
$$
V_{(z|x)} \psi_{\Lambda} = \psi_{\Lambda \cup \{x\}},
$$
$$
V_{(z|x)} V_{(z|y)} = V_{(z|y)} V_{(z|x)}.
$$
\end{proposition}

\begin{proof}
From a simple calculation, we have
\begin{eqnarray*}
\langle V_{(z|x)} e_{i_z} , V_{(z|x)} e_{j_z} \rangle &=&
\delta_{i_z,j_z} \sum_{i_x,j_x} \langle \psi_{xz}(i_x,i_z)e_{i_x},
\psi_{xz}(j_x,i_z)e_{j_x} \rangle \\
&=& \delta_{i_z,j_z} \sum_{i_x} | \psi_{xz}(i_z,i_x)|^2
=\delta_{i_z,j_z} = \langle e_{i_z} , e_{j_z} \rangle.
\end{eqnarray*}
Therefore any $V_{(z|x)}$ is an isometry. Next, we get $V_{(z|x)}
\psi_{\Lambda} =\psi_{\Lambda\cup \{x\}}$. Indeed,
\begin{eqnarray*}
V_{(z|x)} \psi_{\Lambda} &=& V_{(z|x)} (\sum_{\omega_{\Lambda
\backslash \{z\}},i_z} \psi_\Lambda ((\omega_{\Lambda \backslash
\{z\}},i_z))
e_{\omega_{\Lambda \backslash \{z\}}} \otimes e_{i_z}) \\
&=& \sum_{\omega_{\Lambda \backslash \{z\}},i_z} \psi_\Lambda
((\omega_{\Lambda \backslash \{z\}},i_z)) (\sum_{i_x}
\psi_{xz}(i_x,i_z)
e_{\omega_{\Lambda \backslash \{z\}}} \otimes e_{i_x} \otimes e_{i_z} )\\
&=& \sum_{\omega_{\Lambda \backslash \{z\}},i_x,i_z} \psi_{\Lambda
\cup \{x\}} ((\omega_{\Lambda \backslash \{z\}},i_x,i_z))
e_{\omega_{\Lambda \backslash \{z\}}} \otimes e_{i_x} \otimes e_{i_z} \\
&=& \psi_{\Lambda \cup \{x\}}.
\end{eqnarray*}
Finally, we obtain the commutation relation:
\begin{eqnarray*}
V_{(z|x)} V_{(z|y)} e_{i_z} &=&
V_{(z|x)} ( \sum_{i_y} \psi_{yz}(i_y,i_z)e_{i_y} \otimes e_{i_z})\\
&=& \sum_{i_x,i_y}\psi_{xz}(i_x,i_z)\psi_{yz}(i_y,i_z)
e_{i_x}\otimes e_{i_y} \otimes e_{i_z} \\
&=& V_{(z|y)} (\sum_{i_x}\psi_{xz}(i_x,i_z) e_{i_x} \otimes e_{i_z}) \\
&=& V_{(z|y)} V_{(z|x)} e_{i_z}.
\end{eqnarray*}
\end{proof}

For an initial point $x_1 \in L$, we define inductively $\Lambda_1
= \{x_1\}$ and
\begin{equation}\label{eq5.1re}
\Lambda_n =\bar{\Lambda}_{n-1}.
\end{equation}
Then we have the following proposition.

\begin{proposition} Let $\ffi$ be a state defined in Corollary
\ref{existence}, then it is a $d$-Markov chain associated to $\{
\Lambda_n \}$.
\end{proposition}

\begin{proof}
Let $V_n$ be the isometry defined by
\[
V_n = \prod \{ V_{(x|y)} \, : \, x \in \Lambda_n \, , \, y \in
\vec \partial \Lambda_n \, , \, x \sim y \}
\]
where the product is well-defined because, due to Proposition
\ref{commute}, the factors commute. We define the
quasi-conditional expectation with respect to the triple ${\cal
B}_{\Lambda_{n-1}} \subset {\cal B}_{\Lambda_n} \subset {\cal
B}_{\Lambda_{n+1}}$ by
$$
\ce_n (a_{\Lambda_{n+1}})= V_n^* (a_{\Lambda_{n+1}}) V_n
$$
for $a_{\Lambda_{n+1}} \in {\cal B}_{\Lambda_{n+1}}$. Denote
$$
\rho = \frac{1}{d} \langle \sum_{i_{x_1}=1}^d e_{i_{x_1}} , \,
\cdot \, \sum_{j_{x_1}=1}^d e_{j_{x_1}} \rangle.
$$
Then from Proposition \ref{commute}, we have
\begin{eqnarray*}
&& \rho \circ \ce_1 \circ \cdots \circ \ce_n (a_{\Lambda_n}) \\
&=& \frac{1}{d} \langle \sum_{i_{x_1}=1}^d \prod_{x \in \Lambda_2}
V_{(x_1|x)} e_{i_{x_1}}, \ce_{2} \circ \cdots \circ \ce_{n}
(a_{\Lambda_n})
 \prod_{x \in \Lambda_2} V_{(x_1|x)} \sum_{j_{x_1}=1}^d e_{j_{x_1}} \rangle \\
&=& \frac{1}{d} \langle \psi_{\Lambda_2},
\ce_{2} \circ \cdots \circ \ce_{n} (a_{\Lambda_n}) \psi_{\Lambda_2} \rangle \\
&& \vdots \\
&=& \frac{1}{d} \langle \psi_{\Lambda_{n+1}},a_{\Lambda_n} \psi_{\Lambda_{n+1}} \rangle \\
&=& \varphi (a_{\Lambda_n})
\end{eqnarray*}
which implies the assertion. \end{proof}

We don't know if entangled Markov fields are generalized quantum Markov states
or not. But if we assume that
\[
|\psi_{xy}(i,j)|^2 = {1\over d}
\]
for any $x \sim y$ and $1\le i,j \le d$,
we can see the next proposition.

\begin{proposition}\label{prop5.7re} If
$|\psi_{xy}(i,j)|^2 = {1\over d}$
for any $x \sim y$ and $1\le i,j \le d$,
a state $\ffi$ defined in Corollary
\ref{existence} is a generalized quantum Markov state.
\end{proposition}

\begin{proof}
Let $\Lambda_n$ be as in (\ref{eq5.1re}).
Define an isometry $V_n$ from
 ${\cal H}_{\Lambda_{n+2} \backslash \Lambda_n}$
to ${\cal H}_{\Lambda_{n+2} \backslash \Lambda_{n-1}}$ as follows:

\noindent
For $y \in \Lambda_{n-1}$, assume
$\Lambda_n \cap \vec\partial \{y\} = \{ x_1, \ldots , x_m\}$.
We define
\begin{eqnarray*}
&&V_n(e_{i_1}(x_1) \otimes e_{i_2}(x_2) \otimes \cdots \otimes e_{i_k}(x_m))\\
&=&d^{m-1 \over 2} \sum_{j=1}^d \prod_{l=1}^m \psi_{x_{i}y}(i_l,j)
e_{j}(y)\otimes e_{i_1}(x_1) \otimes e_{i_2}(x_2)
\otimes \cdots \otimes e_{i_k}(x_m).
\end{eqnarray*}
Furthermore, we will extend $V_n$ naturally, if it is needed.

Then $V_n$ is an isometry from the definition.
Moreover, $V_n$ satisfies that , for $k \ge n+2$,
\[
V_n (d^{-|\Lambda_{n+1}|/2}\psi_{\Lambda_k \backslash \Lambda_{n}}) =
d^{-|\Lambda_{n}|/2} \psi_{\Lambda_k \backslash \Lambda_{n-1}},
\]
where $V_0 = \emptyset$ and $| \, \cdot\, |$ is the cardinal number.
Indeed,
since $\Lambda_k \backslash \Lambda_n$ is a union of
 $|\Lambda_{n+1}|$ connected sets,
we have
\[
\|\psi_{\Lambda_k \backslash \Lambda_n}\|^2 = d^{|\Lambda_{n+1}|}
\]
by Corollary \ref{existence}.

For $y \in \Lambda_{n}$ and
$k \ge n+1$, let ${\cal V}_n(y,k) = \bigcup \{ x \in \Lambda_l \,|\,
\dis(x,y) = l-n,  n+1 \le l\le k \} \cup \{y\}$, all vertices in
$\Lambda_k \backslash \Lambda_{n-1}$ which connect to $y$
in $\Lambda_k \backslash \Lambda_{n-1}$.
Let $y \in \Lambda_{n}$ and
$\Lambda_{n+1}\cap \vec\partial \{y\} = \{ x_1, \ldots , x_m\}$.
Then one can see that
\begin{eqnarray*}
V_n\left(\bigotimes_{i=1}^m \psi_{{\cal V}_{n+1}(x_i,k)}\right)
= d^{m-1\over 2}\psi_{{\cal V}_n(y,k)}
\end{eqnarray*}
and
\begin{eqnarray*}
&&V_n \left(d^{-|\Lambda_{n+1}|/2}\psi_{\Lambda_k
\backslash \Lambda_{n}}\right)
=
V_n\left( d^{-|\Lambda_{n+1}|/2}
\bigotimes_{x \in \Lambda_{n+1}} \psi_{{\cal V}_{n+1}(x,k)}\right)\\
&=&V_n\left( d^{-|\Lambda_{n}|/2}
\bigotimes_{y \in \Lambda_{n}} \psi_{{\cal V}_{n}(y,k)}\right)
=  d^{-|\Lambda_{n}|/2}  \psi_{\Lambda_k \backslash \Lambda_{n-1}}.
\end{eqnarray*}
Therefore, we have
\[
V_1 \cdot V_2 \cdots V_n \left( d^{-|\Lambda_{n+1}|/2}\psi_{\Lambda_k
\backslash \Lambda_{n}}\right) =
d^{-1/2}\psi_{\Lambda_k}.
\]

Now we define that
\begin{eqnarray*}
\hat\ffi_{\Lambda_n^c} &=& d^{-|\Lambda_{n+1}|} \langle
\psi_{\Lambda_{n+2}
\backslash \Lambda_{n}}, \, \cdot \,
\psi_{\Lambda_{n+2}
\backslash \Lambda_{n}} \rangle \otimes \ffi| \cb_{\Lambda_{n+2}^c} \\
\ce_{\Lambda_n^c} &=& V_n^* \, \cdot \, V_n
\end{eqnarray*}
Then since
\[
\langle
\psi_{\Lambda_{n+2}
\backslash \Lambda_{n}}, a
\psi_{\Lambda_{n+2}\backslash \Lambda_{n}}
\rangle
=
\langle
\psi_{\Lambda_{n+3}
\backslash \Lambda_{n}}, a
\psi_{\Lambda_{n+3}\backslash \Lambda_{n}}
\rangle
\]
for all $a\in \cb_{\Lambda_{n+1}\backslash \Lambda_n}$
from a similar proof of Proposition \ref{treewelldefined},
we have
\[
\hat\ffi_{\Lambda_n^c} | \cb_{\Lambda_{n+1} \backslash \Lambda_n}
= \hat\ffi_{\Lambda_{n+1}^c} \circ \ce_{\Lambda_{n+1}^c}
| \cb_{\Lambda_{n+1} \backslash \Lambda_n}
\]
and for $a \in \cb_{\Lambda_n}$,
\begin{eqnarray*}
&&\hat\ffi_{\Lambda_n^c}\circ \ce_{\Lambda_n^c}
\circ\ce_{\Lambda_{n-1}^c}\circ \cdots \circ \ce_{\Lambda_1^c}
(a)\\
&=&
\langle V_1 \cdot V_2 \cdots V_n \left( d^{-|\Lambda_{n+1}|/2}\psi_{\Lambda_k
\backslash \Lambda_{n}}\right), a
 V_1 \cdot V_2 \cdots V_n \left( d^{-|\Lambda_{n+1}|/2}\psi_{\Lambda_k
\backslash \Lambda_{n}}\right) \rangle
 \\
&=&
\ d^{-1} \langle \psi_{\Lambda_{n+2}}, a\psi_{\Lambda_{n+2}}
\rangle \\
&=&
\ffi(a).
\end{eqnarray*}
This says that $\ffi$ is a generalized quantum Markov state.
\end{proof}

\noindent
{\bf Remark.}
It is not easy to extend the construction of entangled
Markov fields to more general graphs,
because Corollary \ref{existence} does not hold in general.
If we want to make a entangled Markov field on a general graph,
we need the condition that, for each $\Lambda
\subseteq \vec\partial x$,
\[
\sum_{i_x} \prod_{y\in \Lambda, y\sim x} |\psi_{xy}(i_x,i_y)|^2
\]
is constant, i.e. independent of the choice of the $i_y$'s,
as in Proposition \ref{prop5.7re}.
Note
that the last condition is not true in general.

\bigskip
\noindent
{\bf Remark.} From the proved Propositions there arises a natural
question: would the entangled Markov field be a Markov state. Such
a question was not considered in \cite{[AcFi03],[Miyad05]}. Now we
are going to provide an example of the entangled Markov field,
which is not a Markov state.

\bigskip
\noindent
{\bf Example.} For the sake of simplicity, we consider the
simplest tree graph ${\mathbb Z}$ and $\cb_x = M_2$ for all $x\in
{\mathbb Z}$.

Before we see the example, we recall some basic notations about
Markov states on $\cb_{\mathbb Z}$.
A shift $\gamma$ on $\cb_{\mathbb Z}$ is an automorphism on
$\cb_{\mathbb Z}$ defined by
\[
\gamma(X) = I_{M_2} \otimes X
\]
for any $X \in \cb_\Lambda$ and $\Lambda \subfc {\mathbb Z}$.
A shift-invariant Markov state, i.e., $\ffi \circ \gamma = \ffi$,
is generated by a conditional expectation
$\ce : M_2 \otimes M_2 \to M_2$ such that
$ \phi \circ \ce (A\otimes I) = \phi (A)$ for all $A \in M_2$
by the formulation
\[
\ffi (A_1 \otimes A_2 \otimes \cdots \otimes A_n)
=\ffi\circ \ce (A_1\otimes \ce (A_2 \otimes \cdots
\ce(A_{n-1} \otimes A_n)\cdots)).
\]

Then there are  three possible cases of the range of $\ce$.
Namely,
\begin{enumerate}

\item[(i)--] case: ${\rm ran}\ce = \cb_x$.

In this case, $\ffi$ is a product state.

\item[(ii)--] case: ${\rm ran}\ce = {\mathbb C}I$.

In this case, $\ffi$ is also a product state.

\item[(iii)--] case: ${\rm ran}\ce = {\mathbb C} \oplus
{\mathbb C}$.

In this case, we can make a classical shift-invariant Markov chain on
$\bigotimes {\rm ran}\ce = \bigotimes {\mathbb C}\oplus {\mathbb C}$
and $\ffi$ is a canonical extension of this Markov chain
(see \cite{AcLi}).

\end{enumerate}

Now we construct an entangled Markov field which does not belong to
the above three cases.

Put
\begin{eqnarray*}
\psi_{x,y}(1,1) &=& \psi_{x,y}(2,2) = {1\over \sqrt{3}} \\
\psi_{x,y}(1,2) &=& \psi_{x,y}(2,1) = {\sqrt{2} \over \sqrt{3}}
\end{eqnarray*}
for all $x\sim y$. Let $\ffi$ be a entangled Markov field
generated by the above $\psi$ (see Corollary \ref{existence}).
Then one can see that $\ffi$ is shift-invariant. Moreover, $\ffi$
is not a product state, since
\[
\phi(e_{11}) = {1\over 2}, \quad \phi(e_{11} \otimes e_{11}) =
{1\over 6}.
\]
Finally, $\ffi$ is not a canonical extension of classical Markov
chain. Indeed, since $\ffi_{[1,n]}$ is written as a restriction of
vector state on $\cb_{[0,n+1]}$, the density matrix of
$\ffi_{[1,n]}$ is a linear combination of at most 4 one-rank
projections. From the direct calculation, one can get that the
density matrix of $\ffi_{[1,2]}$ is a linear combination of just 4
one-rank projections whose vectors are linearly independent.
Moreover, let $\alpha_n$ be a number of combinations of density
matrix of a classical Markov chain. Then $\alpha_n \to \infty$ or
$\alpha_n =1$ or $\alpha_n =2$. Therefore, $\ffi$ is not  a
canonical extension of classical Markov chain.\\

{\bf Remark.} Let us first recall a definition of entangled state.
Consider ${\cal A}_j$ $(j\in L)$, $C^*$ algebras, here $L$ is a
tree. Denote
\begin{eqnarray*}
&&{\cal S}_{\rm prod}=\overline{\rm Conv} \bigl\{ \bigotimes_{j\in
L} \omega_j \,\, ; \,\, \omega_j \in {\cal
S}({\cal A}_j), \, j \in L \bigr\}, \\
&& {\cal S}_{\L,{\rm prod}}=\overline{\rm Conv} \bigl\{
\omega_{\L}\otimes\omega_{\L^c} \,\, ; \,\, \omega_{\L}\in {\cal
S}(\otimes_{j\in\L}{\cal A}_j), \, \omega_{\L^c}\in {\cal
S}(\otimes_{j\in\L^c}{\cal A}_j)\bigr\},\\
&&{\cal S}_{\mathbb{Z}}=\bigcup_{\L\subset L:\atop
\L\sim\mathbb{Z}}{\cal S}_{\L,{\rm prod}},
\end{eqnarray*}
here by $\L\sim \mathbb{Z}$  we mean an isomorphism (i.e. a 1-1
mapping which preserves edges and connected components) of a
subgraph $\L\subset L$ to the integer lattice $\mathbb{Z}$.

 A state $\omega\in {\cal S}(\otimes_{j\in L}{\cal A}_j)$ is
said to be {\it entangled} (see \cite{[AcMaOh]} (resp. {\it
$\mathbb{Z}$-entangled}) if $\omega\notin {\cal S}_{\rm prod}$
(resp. $\omega\notin {\cal S}_{\mathbb{Z}}$). One can see that any
$\mathbb{Z}$-entangled state is entangled, but the converse is not
true. In \cite{[AcMaOh]} it has been established that entangled
quantum Markov states on $\mathbb{Z}$ are entangled.

From the definition given above we can prove

\begin{theorem}
Let $\ffi$ be a state on $\cb_L$. The following assertions hold:
\begin{enumerate}
\item[(i)] If for some $\L$ with $\L\sim \mathbb{Z}$ the
restriction of $\ffi$ to the $C^*$-subalgebra $\cb_{\L}$ is
entangled, then $\ffi$ is also entangled on $\cb_L$; \item[(ii)]
If for any $\L$ with $\L\sim \mathbb{Z}$ the restriction of $\ffi$
to the $C^*$-subalgebra $\cb_{\L}$ is ${\mathbb Z}$-entangled, then $\ffi$ is
$\mathbb{Z}$-entangled on $\cb_L$.
\end{enumerate}
\end{theorem}


\section{$d$-Markov chains on Cayley trees}\label{dfcayley}

In this section, we consider a particular case of tree, so called
Cayley tree. Over such a tree we are going give a construction of
$d$-Markov chains.

Recall that a Cayley tree $\Gamma^k$ of order $k \ge 1$ is an
infinite tree whose each vertices have exactly $k+1$ edges. If we
cut away  an edge $\{x,y\}$ of the tree $\Gamma^k$, then
$\Gamma^k$ splits into connected components, called semi-infinite
trees with roots $x$ and $y$, which will be denoted respectively
by $\Gamma^k(x)$ and $\Gamma^k(y)$. If we cut away from $\Gamma^k$
the origin $O$ together with all $k+1$ nearest neighbor vertices,
in the result we obtain $k+1$ semi-infinite $\Gamma^k(x)$ trees
with $x \in S_0 = \{ y \in \Gamma^k \, : \, \dis (O ,y) =1\}$,
where $\dis$ is a distance of vertices introduced in Sect. \ref{graphs}.
Hence
we have
\[
\Gamma^k = \bigcup_{x\in S_0} \Gamma^k(x) \cup \{ O\}.
\]

Therefore, in the sequel we will consider semi-infinite Cayley tree
$\Gamma^k(x_0) = (L,E)$ with the root $x_0$. Let us set
\[
W_n = \{ x\in L \, : \, \dis (x,x_0) = n\} , \qquad \Lambda_n =
\bigcup_{k=0}^n W_k, \qquad E_n = \{ \{x,y\} \in E \, : \, x,y \in
\Lambda_n\}.
\]
In the following, we will construct examples of
$d$-Markov chains on semi-infinite Cayley trees,
that is, we construct a sequence of quasi-conditional expectations $\ce_n$
with respect to $\cb_{\Lambda_{n-1}} \subset \cb_{\Lambda_{n}} \subset
\cb_{\Lambda_{n+1}}$ and an initial state $\rho$, and define
\[
\ffi = \lim \rho \circ \ce_0 \circ \ce_1 \circ \cdots \circ \ce_n.
\]
For this, we use some operators $V_n \in \cb_{\Lambda_{n+1}\backslash
\Lambda_{n-1}}$ and define $\ce_n = {\rm Tr}_{\Lambda_n}(V_n \cdot V_n^*)$,
where ${\rm Tr}_{\Lambda_n}$ is a normalized trace from $\cb_L$ to
$\cb_{\Lambda_n}$.

Denote
$$
S(x)=\{y\in W_{n+1} :  x\sim y \}, \ \ x\in W_n, $$ this set is
called a set of {\it direct successors} of $x$.

From these one can see that
\begin{eqnarray}\label{v-n}
&&\L_m=\L_{m-2}\cup\bigg(\bigcup_{x\in W_{m-1}}\{x\cup S(x)\}\bigg)\\
\label{Ln} &&E_m\setminus E_{m-1}=\bigcup_{x\in
W_{m-1}}\bigcup_{y\in S(x)}\{\{ x,y\} \}
\end{eqnarray}

Now we are going to introduce a coordinate structure in $\G^k(x_0)$.
Every vertex $x$ (except for $x_0$) of $\G^k(x_0)$ has coordinates
$(i_1,\dots,i_n)$, here $i_m\in\{1,\dots,k\}$, $1\leq m\leq n$ and
for the vertex $x_0$ we put $\emptyset$.
  Namely, the symbol $\emptyset$
constitutes level 0 and the sites $(i_1,\dots,i_n)$ form level $n$
of the lattice. In this notation for $x\in \G^k(x_0)$,
$x=(i_1,\dots,i_n)$ we have
$$
S(x)=\{(x,i):\ 1\leq i\leq k\},
$$
here $(x,i)$ means that $(i_1,\dots,i_n,i)$.


Then for $1\le i \le k$, we define a shift $\gamma_i$ by
\[
\gamma_i(x) = (i,x) = (i, i_1, \ldots , i_n).
\]
Now we can consider this shift as a shift homomorphism on
${\mathcal B}_L$, that is, for any $a_x \in \cb_x$, we consider
$\gamma_i(a_x) \in \cb_{(i,x)}$.

Let be given a positive operator $w_{0}\in {\cal B}_{x_0,+}$ and two
family of operators $\{K_{<x,y>}\in {\cal B}_{\{x,y\}}\}_{\{x,y\}\in
E}$,  $\{h_x\in {\cal B}_{x,+}\}_{x\in L}$  such that
\begin{eqnarray}\label{eq1}
&& \tr ( w_0 h_0 ) =1 \\
\label{eq2}&& \tr_x\bigg(\prod_{i=1}^kK_{<x,(x,i)>} \prod_{i=1}^k
h_{(x,i)} \prod_{i=1}^k K^*_{<x,(x,k+1-i)>}\bigg)=h_x, \ \
\textrm{for every} \ \  x\in L,
\end{eqnarray}
where ${\rm Tr}_{\Lambda} : {\cal B}_L \to {\cal B}_\Lambda$ is a
normalized partial trace for any $\Lambda \subf L$ and $\tr$ is a
normalized trace on ${\cal B}_L$.

 Note that if $k=1$ and $h_x=I$ for all $x\in V$, then we get
conditional amplitudes introduced by L.Accardi \cite{[AcFr80]}.

Denote
\begin{equation}\label{Kn1}
K_{n}=w^{1/2}_{0}\prod_{\{x,y\}\in E_1}K_{<x,y>} \prod_{\{x,y\}\in
E_2\setminus E_1}K_{<x,y>}\cdots \prod_{\{x,y\}\in E_n\setminus
E_{n-1}}K_{<x,y>}\prod_{x\in W_n}h^{1/2}_x,
\end{equation}
where by definition we put
\begin{eqnarray}\label{Kn2}
\prod_{\{x,y\}\in E_m\setminus E_{m-1}}K_{<x,y>} :=\prod_{x\in
W_{m-1}}\prod_{i=1}^kK_{<x,(x,i)>}
\end{eqnarray}

Now define
\begin{equation}\label{Wn}
\cw_{n]}=K_{n}^*K_{n}.
\end{equation}

It is clear that $\cw_{n]}$ is positive.

Now we are ready to define a positive functional $\ffi^{(n)}$ on
$\cb_{\Lambda_n}$ by
\begin{eqnarray}\label{ffinofx}
\ffi^{(n)}(a)=\tr(\cw_{n+1]}(a\otimes\id_{W_{n+1}})),
\end{eqnarray}
for every $a\in \cb_{\Lambda_n}$, where
$\id_{W_n+1}=\bigotimes\limits_{y\in W_{n+1}}\id$.

To get a state $\ffi$ on $\cb_L$, i.e. on the infinite-volume tree,
by means of $\ffi^{(n)}$ such that
$\ffi\lceil_{\cb_{\L_n}}=\ffi^{(n)}$, we need to impose some
constrains to the boundary conditions $\big\{w_0,\{h_x\}\big\}$ so
that the functionals $\{\ffi^{(n)}\}$ satisfy the compatibility
condition, i.e.
\begin{eqnarray}\label{compatibility}
\ffi^{(n+1)}\lceil_{\cb_{\L_n}}=\ffi^{(n)}.
\end{eqnarray}

\begin{theorem}\label{proj} Let the boundary conditions $w_{0}\in {\cal
B}_{(0),+}$ and $\{h_x\in {\cal B}_{x,+}\}_{x\in L}$ satisfy
\eqref{eq1} and \eqref{eq2}. Then the functionals $\{\ffi^{(n)}\}$
satisfy the compatibility condition \eqref{compatibility}. Moreover,
there is a unique state $\ffi$ on $\cb_L$ such that
$\ffi=w-\lim_{n\to\infty}\ffi^{(n)}$.
\end{theorem}

\begin{proof} Let us check the equality \eqref{compatibility}. From \eqref{Kn1} one has
\begin{eqnarray*}
\cw_{n+2]}&=& \prod_{x\in W_{n+2}}h^{1/2}_x\bigg(\prod_ {\{x,y\}\in
E_{n+2}\setminus
E_{n+1}}K_{<x,y>}\bigg)^*\cdots\bigg(\prod_{\{x,y\}\in
E_1}K_{<x,y>}\bigg)^*w_0\\
&&\prod_{\{x,y\}\in E_1}K_{<x,y>}\cdots\prod_ {\{x,y\}\in
E_{n+2}\setminus E_{n+1}}K_{<x,y>} \prod_{x\in W_{n+2}}h^{1/2}_x
\end{eqnarray*}
Then for any $a\in \cb_{\Lambda_n}$ we find
\begin{eqnarray*}
\ffi^{(n+1)}(a\otimes\id_{W_{n+1}})&=&\tr(\cw_{n+2]}(a\otimes\id_{W_{n+1}}\otimes\id_{W_{n+2}}))\\
&=& \tr \left(\prod_{x\in W_{n+2}}h^{1/2}_x\bigg(\prod_ {\{x,y\}\in
E_{n+2}\setminus
E_{n+1}}K_{<x,y>}\bigg)^*\cdots\right.\\
&& \cdots\bigg(\prod_{\{x,y\}\in
E_1}K_{<x,y>}\bigg)^*w_0\prod_{\{x,y\}\in
E_1}K_{<x,y>}\cdots\\
&&\left. \cdots\prod_ {\{x,y\}\in E_{n+2}\setminus E_{n+1}}K_{<x,y>}
\prod_{x\in
W_{n+2}}h^{1/2}_x(a\otimes\id_{W_{n+1}}\otimes\id_{W_{n+2}})\right)\\
&=& \tr \left(\bigg(\prod_ {\{x,y\}\in E_{n+1}\setminus
E_{n}}K_{<x,y>}\bigg)^*\cdots\bigg(\prod_{\{x,y\}\in
E_1}K_{<x,y>}\bigg)^*w_0\right.\\
&& \prod_{\{x,y\}\in E_1}K_{<x,y>}\cdots\prod_ {\{x,y\}\in
E_{n+1}\setminus
E_{n}}K_{<x,y>}(a\otimes\id_{W_{n+1}}) \\
&&\left. \prod_ {\{x,y\}\in E_{n+2}\setminus E_{n+1}}K_{<x,y>}
\prod_{x\in W_{n+2}}h_x\bigg(\prod_ {\{x,y\}\in E_{n+2}\setminus
E_{n+1}}K_{<x,y>}\bigg)^*\right)
\end{eqnarray*}

We know that for different $x$ and $x'$ taken from $W_{n+1}$ the
algebras $\cb_{x\cup S(x)}$ and  $\cb_{x'\cup S(x')}$ commute,
therefore from \eqref{Kn2} one finds
\begin{eqnarray*}
&& \prod_{\{x,y\}\in E_{n+2}\setminus E_{n+1}}K_{<x,y>} \prod_{x \in
W_{n+2}} h_x \bigg(\prod_{\{x,y\} \in E_{n+2}\setminus
E_{n+1}}K_{<x,y>}\bigg)^* =\\&& \prod_{x\in W_{n+1}}\prod_{i=1}^k
K_{<x,(x,i)>} \prod_{i=1}^k h_{(x,i)}
\prod_{i=1}^kK^*_{<x,(x,k+1-i)>}.
\end{eqnarray*}
and
\begin{eqnarray*}
&& \tr_{W_{n+1}]}\left(\prod_{\{x,y\}\in E_{n+2}\setminus
E_{n+1}}K_{<x,y>} \prod_{x \in W_{n+2}} h_x \bigg(\prod_{\{x,y\} \in
E_{n+2}\setminus
E_{n+1}}K_{<x,y>}\bigg)^*\right) \\
&=& \prod_{x\in W_{n+1}}\tr_{\{x\}}\left(\prod_{i=1}^k K_{<x,(x,i)>}
\prod_{i=1}^k h_{(x,i)} \prod_{i=1}^kK^*_{<x,(x,k+1-i)>}\right).
\end{eqnarray*}

Hence from the condition \eqref{eq2} we find
\begin{eqnarray*}
\ffi^{(n+1)}(a\otimes\id_{W_{n+1}})&=& \tr \left(\bigg(\prod_
{\{x,y\}\in E_{n+1}\setminus
E_{n}}K_{<x,y>}\bigg)^*\cdots\bigg(\prod_{\{x,y\}\in
E_1}K_{<x,y>}\bigg)^*w_0\right.\\
&& \prod_{\{x,y\}\in E_1}K_{<x,y>}\cdots\prod_ {\{x,y\}\in
E_{n+1}\setminus
E_{n}}K_{<x,y>}(a\otimes\id_{W_{n+1}}) \\
&&\left. \prod_{x\in W_{n+1}}Tr_{\{x\}}\left(\prod_{i=1}^k
K_{<x,(x,i)>} \prod_{i=1}^k
h_{(x,i)} \prod_{i=1}^kK^*_{<x,(x,k+1-i)>}\right)\right)\\
&=& \tr \left(\bigg(\prod_ {\{x,y\}\in E_{n+1}\setminus
E_{n}}K_{<x,y>}\bigg)^*\cdots\bigg(\prod_{\{x,y\}\in
E_1}K_{<x,y>}\bigg)^*w_0\right.\\
&& \left. \prod_{\{x,y\}\in E_1}K_{<x,y>}\cdots\prod_ {\{x,y\}\in
E_{n+1}\setminus E_{n}}K_{<x,y>}(a\otimes\id_{W_{n+1}})\prod_{x\in
W_{n+1}}h_x\right)\\
&=& \tr \left(\prod_{x\in W_{n+1}}h^{1/2}_x\bigg(\prod_ {\{x,y\}\in
E_{n+1}\setminus
E_{n}}K_{<x,y>}\bigg)^*\cdots\bigg(\prod_{\{x,y\}\in
E_1}K_{<x,y>}\bigg)^*w_0\right.\\
&& \left. \prod_{\{x,y\}\in E_1}K_{<x,y>}\cdots\prod_ {\{x,y\}\in
E_{n+1}\setminus E_{n}}K_{<x,y>}\prod_{x\in
W_{n+1}}h^{1/2}_x(a\otimes\id_{W_{n+1}})\right) \\
&=&
\tr(\cw_{n+1]}(a\otimes\id_{W_{n+1}}))\\
&=&\ffi^{(n)}(a).
\end{eqnarray*}

From \eqref{eq1}, one can show that $\ffi^{(n)}$ is a state, i.e.
$\ffi^{(n)}(\id_{\Lambda_n})=1$.
\end{proof}

Assume that $h_x$ is invertible for all $x \in L$ and define
\begin{eqnarray}\label{Fn}
\ce_{n}(a) &=& \tr_{n]}\bigg(\prod_{x\in W_n} h_x^{-1/2}
\prod_{\{x,y\}\in E_{n+1}\setminus E_{n}}K_{<x,y>}\prod_{x\in
W_{n+1} }h_{x}^{1/2}
a \\
&& \times \prod_{x\in W_{n+1}}h_{x}^{1/2} \bigg(\prod_{\{x,y\}\in
E_{n+1}\setminus E_{n}}K_{<x,y>}\bigg)^* \prod_{x\in W_n}
h_x^{-1/2} \bigg)
\end{eqnarray}
for each $n \ge 0$ and $a \in \cb_{\Lambda_{n+1}}$. Similar to the
above proof, we get that $\ce_{n}$ is a quasi-conditional
expectation with respect to the triple $\cb_{\Lambda_{n-1}}
\subset \cb_{\Lambda_{n}} \subset \cb_{\Lambda_{n+1}}$. One can
see that
\begin{equation}\label{mar1}
\ffi_n (a)=\tr(h_0^{1/2} w_0 h_0^{1/2} \ce_0\circ
\ce_1\circ\cdots\circ \ce_{n-1}\circ \ce_n(a)).
\end{equation}
Therefore, according to Theorem \ref{proj} we can define a
$d$-Markov chain on $\cb_L$ by $\ffi = \lim \ffi_n$ in the weak-$*$
topology. Note that, in classical setting, similar construction were
considered in \cite{Spa}.

If $h_x =h $ and $\kxy =K$, for all $x \in L$ and $\{x,y\} \in E$,
and $w_0$ satisfies the initial condition
\begin{equation}\label{initial}
\tr_{(i)} \left( w_0 \prod_{j=1}^k K_{<0,j>} \prod_{j=1}^k h_j
(\prod_{j=1}^k K_{<0,j>} )^*   \right) = h_i^{1/2} w_0 h_i^{1/2},
\end{equation}
where $K_{<0,j>}$ means $K_{<x_0, (j)>}$,
$\ffi$ is shift-invariant for $\gamma_i$. Indeed, since
(\ref{initial}) means
\[
\tr(h_0^{1/2} w_0 h_0^{1/2} \ce_0 (\, \cdot \,)) =\tr(h_i^{1/2}
w_0 h_i^{1/2} \, \cdot \,)
\]
on $\cb_{i}$, we have
\begin{eqnarray*}
\ffi_n (\gamma^i(a))&=&\tr(h_0^{1/2} w_0 h_0^{1/2} \ce_0\circ
\ce_1\circ\cdots\circ
\ce_{n-1}\circ \ce_n(\gamma^i(a)))\\
&=& \tr(h_i^{1/2} w_0 h_i^{1/2} \ce_1\circ \ce_2 \circ\cdots\circ
\ce_{n-1}\circ \ce_n(\gamma^j(a)))\\
&=& \tr(h_0^{1/2} w_0 h_0^{1/2} \ce_0\circ \ce_1\circ\cdots\circ
\ce_{n-2}\circ \ce_{n-1}(a)) = \ffi(a)
\end{eqnarray*}
for all $a \in \cb_{\Lambda_{n-1}}$. In the third equation, we use
$h_0 = h_i =h$ and $\kxy =K$.



\section{Example of $d$-Markov chain on Cayley tree}\label{exam1}

In this and next sections, we provide more concrete examples of
$d$-Markov chains on Cayley tree. For the sake of simplicity we
consider a semi-infinite Cayley tree $\G^2(x_0)=(L,E)$ of order 2
so that $d=2$.
Our starting $C^{*}$-algebra is the same $\cb_L$ but with
$\cb_{x}=M_{2}(\bc)$ for $x\in L$. By $e_{ij}^{(x)}$ we denote the
standard matrix units of $\cb_{x} = M_2(\bc)$.

For every edge $\{x,y\}\in E$ put
\begin{equation}\label{1Kxy1}
K_{<x,y>}=\exp\{\b H_{<x,y>}\}, \ \ \b\in\br
\end{equation}
where
\begin{equation}\label{1Hxy1}
H_{<x,y>}=e_{12}^{(x)}\otimes e_{21}^{(y)}+e_{21}^{(x)}\otimes
e_{12}^{(y)}.
\end{equation}

Now we are going to find a solution  $\{h_x\}$ and $w_0$ of
equations \eqref{eq1}, \eqref{eq2} for the defined
$\{K_{<x,y>}\}$. Note that from \eqref{1Kxy1},\eqref{1Hxy1} for
every $K_{<x,y>}$ one can see that
\begin{eqnarray}\label{1Kxy*}
&&K_{<x,y>}=K^*_{<x,y>}
\end{eqnarray}
for all $\{x,y\}\in E$.

Assume that $h_x=\a I$ for every $x\in V$. Hence, thanks to
\eqref{1Kxy*}, the equations \eqref{eq1},\eqref{eq2} can be
rewritten as follows
\begin{eqnarray}\label{eq11}
&&\a\tr_{0}( w_0 )= 1 \\
\label{eq21}&&
\a^2\tr_x\bigg(K_{<x,(x,1)>}K^2_{<x,(x,2)>}K_{<x,(x,1)>} \bigg)=\a
I, \ \ \textrm{for every} \ \  x\in L.
\end{eqnarray}

One can see that
\begin{eqnarray}\label{1Hxy2}
&&H_{<x,y>}^{2n}=H_{<x,y>}^2=e_{11}^{(x)}\otimes
e_{22}^{(y)}+e_{22}^{(x)}\otimes e_{11}^{(y)}\\
\label{1Hxy3} &&H_{<x,y>}^{2n-1}=H_{<x,y>}
\end{eqnarray}
for every $n\in\bn$. Then we get
\begin{eqnarray}\label{1Kxy3}
K_{<x,y>}^2 &=& I+(\sinh 2\b)H_{<x,y>}+(\cosh 2\b -1) H_{<x,y>}^2 \\
\tr_x (K_{<x,y>}^2) &=& {\cosh2 \b +1 \over 2} I= \cosh^2 \b  I
\nonumber
\end{eqnarray}
for every $\{x,y\}\in E$. Hence, for $x\in L$ and $y,z \in S(x)$,
one finds
\begin{eqnarray*}
\tr_x\left(K_{<x,y>}K^2_{<x,z>}K_{<x,y>} \right) &=&
\tr_x\left(K_{<x,y>} \tr_{xy}(K^2_{<x,z>})K_{<x,y>} \right)\\
&=& (\cosh^2 \beta)
\tr_x\left(K_{<x,y>}^2 \right) \\
&=& (\cosh^4 \beta) I.
\end{eqnarray*}
Therefore we obtain $\a = \cosh^{-4} \beta$ and $\tr(w_0) =
\cosh^4 \beta$.

Next, consider the initial condition (\ref{initial}).
For convenience, we will write $K_{<0,1>}$ for $K_{<x_0,(1)>}$, for example.
Since
\begin{eqnarray*}
 \tr_1\left(w_0 K_{<0,1>}K^2_{<0,2>}K_{<0,1>} \right)
&=&
\tr_1\left(w_0 K_{<0,1>} \tr_{0,1}(K^2_{<0,2>})K_{<0,1>} \right)\\
&=& (\cosh^2 \beta) \tr_1\left(w_0K_{<0,1>}^2 \right),
\end{eqnarray*}
by putting $w_0 = \sum_{i,j=1,2} a_{ij} e_{ij}^{0}$, thanks to
(\ref{1Kxy3}) we have
\begin{eqnarray*}
&& \tr_1\left(w_0 K_{<0,1>}K^2_{<0,2>}K_{<0,1>} \right) \\
&=& {\cosh^2 \beta\over 2} \left( (a_{11} +a_{22})I + (\cosh 2\b
-1)(a_{11} e_{22}+ a_{22}e_{11}) + (\sinh 2\b)(a_{12}e_{12} +
a_{21}e_{21}) \right).
\end{eqnarray*}
This is equal to $(\cosh^4\beta)w_0$ from (\ref{initial}).
Therefore we have the solution $w_0 =I$. Hence, $\ffi$
generated by the above notations is $\gamma_1$-invariant
$d$-Markov chain. Similarly, it is easily seen that $\ffi$ is also
$\gamma_2$-invariant.

Finally we show the clustering property. Recall that a state $\ffi$
on $\cb_L$ satisfies {\it the clustering property} w.r.t. $\gamma_i$
if and only if
\[
\lim_{n\to\infty} \ffi(\gamma_i^n(a)b) = \ffi(a) \ffi(b)
\].

\begin{theorem}\label{cluster1}
A state $\ffi$ generated by the above notations is $\gamma_1$ and
$\gamma_2$-invariant and satisfies clustering property w.r.t.
$\gamma_i$, $i=1,2$.
\end{theorem}

\begin{proof}
The first assertion is already proven in above.

To show the clustering property, it is enough to prove for any $a
\in {\cal B}_0 = M_2({\mathbb C})$
\begin{eqnarray*}
\lim_{n\to\infty}  \ce_0\circ \ce_1 \circ \cdots \circ \ce_{n-1}
\circ \ce_n (\gamma_1^{n+1}(a)) = \ffi(a) I.
\end{eqnarray*}
Indeed, for $a,b \in \cb_0$, we have
\begin{eqnarray*}
\lim_{n\to \infty} \ffi(\gamma_1^{n}(a) b) &=& \lim_{n\to\infty}
\tr\left( h_0^{1/2} w_0 h_0^{1/2} \ce_0( \ce_1 \circ \cdots \circ
\ce_{n-1} \circ \ce_n (\gamma_1^{n+1}(a))b)\right) \\
&=&\ffi(a)  \tr\left( h_0^{1/2} w_0 h_0^{1/2} \ce_0(b)\right) =
\ffi(a) \ffi(b).
\end{eqnarray*}

Assume $\gamma_1^{n+1}(a) \in \cb_{y}$ and $y,z \in S(x)$, then
essentially, we can restrict $\ce_n$ to $\ce_n | \cb_{x,y,z}$. From
a simple calculation, we have
\begin{eqnarray*}
\tr_x (\kxy \kxz e_{11}^{(y)} \kxz \kxy)&=&
\tr_x(\kxy e_{11}^{(y)} \tr_{xy} (\kxz^2) \kxy) \\
&=& (\cosh^2 \beta)\tr_x(\kxy e_{11}^{(y)}  \kxy) \\
&=& {\cosh^4 \beta \over 2} I.
\end{eqnarray*}
Similarly, we get
\begin{eqnarray*}
\tr_{x} (\kxy \kxz e_{22}^{(y)} \kxz \kxy) &=&{\cosh^4 \beta \over 2} I, \\
\tr_{x} (\kxy \kxz e_{12}^{(y)} \kxz \kxy) &=&\cosh^2 \beta \si e_{12}^{(x)}, \\
\tr_{x} (\kxy \kxz e_{21}^{(y)} \kxz \kxy) &=& \cosh^2 \beta \si
e_{21}^{(x)}.
\end{eqnarray*}
Therefore, we obtain that
\begin{eqnarray*}
\lim_{n\to\infty} \ce_0\circ \ce_1 \circ \cdots \circ \ce_{n-1}
\circ \ce_n (\gamma_1^{n+1}(a)) =\tr(a) = \ffi(a) I
\end{eqnarray*}
which implies the assertion.

Similarly, one can prove that $\ffi$ satisfies clustering property
w.r.t. $\gamma_2$.
\end{proof}


\section{Another example of $d$-Markov chain on Cayley tree}\label{exam2}

Now consider the next example. For every edge $\{x,y\} \in E$ put
\[
K_{<x,y>}=\exp\{\b P_{<x,y>}\}, \ \ \b\in\bc
\]
where
\[
P_{<x,y>}=e_{11}^{(x)}\otimes e_{11}^{(y)}+e_{22}^{(x)}\otimes
e_{22}^{(y)}.
\]
Explicitly, we can write
\begin{equation}\label{eq10.1}
K_{<x,y>} = I + (e^\b -1)  P_{<x,y>}.
\end{equation}

Now we are going to find a solution  $\{h_x\}$ and $w_0$ of
equations \eqref{eq1}, \eqref{eq2} for the defined
$\{K_{<x,y>}\}$. Note that for every $K_{<x,y>}$ and $\kxz$, one
can see that
\begin{eqnarray*}
&&K_{<x,y>}=K^*_{<x,y>}, \\
&&\kxy \kxz = \kxz \kxy.
\end{eqnarray*}

Assume that $h_x=\a I$ for every $x\in V$. Hence, thanks to the
above equations, the equations \eqref{eq1},\eqref{eq2} can be
rewritten as follows
\begin{eqnarray}\label{eq112}
&&\a\tr_{0}( w_0 )= 1 \\
\label{eq212}&& \a^2\tr_x\bigg(K_{<x,(x,1)>}^2 K^2_{<x,(x,2)>}
\bigg)=\a I, \ \ \textrm{for every} \ \  x\in L.
\end{eqnarray}
From $\tr_{xy} (\kxz^2) = {e^{2\b} +1 \over 2} I$, we have
\[
\tr_x\left (K_{<x,(x,1)>}^2 K^2_{<x,(x,2)>} \right) = {(e^{2\b}
+1)^2 \over 4} I.
\]
Hence we obtain
\[
\alpha = { 4 \over (e^{2\b} +1)^2 }
\]
and $\tr (w_0)  = {(e^{2\b} +1)^2 / 4}$.

Next, consider the initial condition (\ref{initial}).
Since
\begin{eqnarray*}
 \tr_1\left(w_0 K^2_{<0,1>}K^2_{<0,2>} \right)
&=&
\tr_1\left(w_0 K^2_{<0,1>} \tr_{0,1}(K^2_{<0,2>}) \right)\\
&=& {e^{2\b} +1 \over 2} \tr_1\left(w_0K_{<0,1>}^2 \right),
\end{eqnarray*}
by putting $w_0 = \sum_{i,j=1,2} a_{ij} e_{ij}^{0}$, thanks to
(\ref{eq10.1}) we have
\begin{eqnarray*}
&& \tr_1\left(w_0 K_{<0,1>}^2 K^2_{<0,2>} \right) \\
&=&
 {e^{2\b} +1 \over 4}
\left( (e^{2\b} a_{11} + a_{22})e_{11} +(a_{11} e^{2\b} + a_{22})
e_{22}) \right)
\end{eqnarray*}
This is equal to ${(e^{2\b} +1)^2 \over 4} w_0$ from
(\ref{initial}). Threrfore we have the solution $w_0 =I$.
Therefore, $\ffi$ generated by the above notations is
$\gamma_1$-invariant $d$-Markov chain. Similarly, it is easily seen
that $\ffi$ is also $\gamma_2$-invariant.

Finally, we show the clustering property.

\begin{theorem}
The above $\ffi$ is $\gamma_1$ and $\gamma_2$-invariant and
satisfies clustering property w.r.t. $\gamma_i$, $i=1,2$.
\end{theorem}

\begin{proof}
The first assertion is already proven in above.

This proof is similar to Theorem \ref{cluster1}, and we need to
show
\begin{eqnarray*}
\lim_{n\to\infty}  \ce_0\circ \ce_1 \circ \cdots \circ \ce_{n-1}
\circ \ce_n (\gamma_1^{n+1}(a)) = \ffi(a) I.
\end{eqnarray*}
for $a \in \cb_0$. To see this, we make following lists for $x\in
L$ and $y,z \in S(x)$:
\begin{eqnarray*}
\tr_{x} (\kxy \kxz e_{11}^{(z)} \kxz \kxy) &=&
 {e^{2\beta}(e^{2\beta}+1) \over 4} e_{11}^{(x)} + {e^{2\beta}+1 \over 4}e_{22}^{(x)}, \\
\tr_{x} (\kxy \kxz e_{22}^{(z)} \kxz \kxy) &=&
 {e^{2\beta}+1 \over 4}e_{11}^{(x)} +
 {e^{2\beta}(e^{2\beta}+1) \over 4} e_{22}^{(x)}. \\
\tr_{x} (\kxy \kxz e_{12}^{(z)} \kxz \kxy) &=& 0 \\
\tr_{x} (\kxy \kxz e_{21}^{(z)} \kxz \kxy) &=& 0.
\end{eqnarray*}
As in the classical Markov chain case, we can prove that
\begin{eqnarray*}
\lim_{n\to\infty}  \ce_0\circ \ce_1 \circ \cdots \circ \ce_{n-1}
\circ \ce_n (\gamma_1^{n+1}(a)) = \tr(a) = \ffi(a) I.
\end{eqnarray*}
which proves the theorem.

Similarly, we can prove that $\ffi$ satisfies clustering property
w.r.t. $\gamma_2$.
\end{proof}

\section{Conclusions}

Let us note that a first attempt of consideration of quantum
Markov fields began in \cite{[AcFi01a], [AcFi01b]} for the regular
lattices (namely for $\mathbb{Z})$. But there, concrete examples
of such fields were not given. In the present paper we have
extended a notion of generalized quantum Markov states to fields,
i.e. to graphs with an hierarchy property. Here such states have
been considered on discrete infinite tensor products of
$C^*$--algebras over trees. A tree structure of graphs allowed us
to give a construction an entangled Markov field, which
generalizes the construction of \cite{[AcFi03]} to trees. It has
been shown that such states are d-Markov chains and, in special
cases, they are generalized quantum Markov states.

As well as, we have considered a particular case of tree, so
called Cayley tree. Over such a tree we gave a construction of
$d$-Markov chains, and some more concrete examples of such chains
were provided, which are shift invariant and have the clustering
property. Note that  $d$-Markov chains describe ground states of
quantum systems over trees. Certain particular examples of such
systems were considered in \cite{aklt},\cite{fannes}. As well as,
such shift invariant $d$-Markov chains can be  also considered as
an extension of $C^*$-finitely correlated states defined in
\cite{fannes2} to the Cayley trees.


\section*{Acknowledgement} The second named author (H.O.) is
partially supported by Grant-in-Aid for JSPS Fellows 19$\cdot$2166.
The third named author (F.M.) thanks the grant FRGS0308-91 of MOHE.
Finally, the authors also would like to thank to a referee for his
useful suggestions which allowed us to improve the content of the
paper.


\begin{thebibliography}{19}

\bibitem{[Ac74f]}
Accardi L.: On the noncommutative Markov property, \emph{Func. Anal.
Appl.}, {\bf 9} (1975) 1--8.

\bibitem{[AcFi03]}
Accardi L., Fidaleo F.: Entangled Markov chains. \textit{Annali di
Matematica Pura e Applicata}, {\bf 184}(2005), 327--346.


\bibitem{[AcFi01a]}
Accardi L., Fidaleo F.: Quantum Markov fields, \textit{Inf. Dim.
Analysis, Quantum Probab. Related Topics} {\bf 6} (2003) 123--138.

\bibitem{[AcFi01b]}
Accardi L., Fidaleo F.: On the structure of quantum Markov fields,
Proceedings Burg Conference 15--20 March 2001, W. Freudenberg
(ed.), World Scientific, QP--PQ Series 15 (2003) 1--20

\bibitem{[AcFiMu07]}
Accardi L., Fidaleo F. Mukhamedov, F.: Markov states and chains on
the CAR algebra, \textit{Inf. Dim. Analysis, Quantum Probab. Related
Topics} {\bf 10} (2007), 165--183.


\bibitem{[AcFr80]}
Accardi L., Frigerio A.: Markovian cocycles, \emph{Proc. Royal Irish
Acad.} {\bf 83A} (1983) 251-263.

\bibitem{AcLi}
Accardi L., Liebscher V.: Markovian KMS-states for one-dimensional
spin chains, \emph{Infin. Dimens. Anal. Quantum Probab. Relat.
Top.,} {\bf 2}(1999) 645-661.

\bibitem{[AcMaOh]}
Accardi L., Matsuoka T., Ohya M.: Entangled Markov chains are indeed
entangled, \textit{Inf. Dim. Analysis, Quantum Probab. Related
Topics} {\bf 9} (2006), 379--390.


\bibitem{aklt} Affleck L, Kennedy E., Lieb E.H., Tasaki H.:
Valence bond ground states in isortopic quantum antiferromagnets,
\emph{Commun. Math. Phys.} {\bf 115} (1988), 477--528.


\bibitem{[Dobr68a]}
Dobrushin R.L.: Description of Gibbsian Random Fields by means of
conditional probabilities, \emph{Probability Theory and
Applications} {\bf 13}(1968) 201--229

\bibitem{fannes}
Fannes M., Nachtergaele B. Werner R. F.: Ground states of VBS models
on Cayley trees, \emph{J. Stat. Phys.} {\bf 66} (1992) 939--973.

\bibitem{fannes2}
Fannes M., Nachtergaele B. Werner R. F.: Finitely correlated states
on quantum spin chains, \emph{Commun. Math. Phys.} {\bf 144} (1992)
443--490.


\bibitem{[Liebs99]}
Liebscher V.: Markovianity of quantum random fields, Proceedings
Burg Conference 15--20 March 2001, W. Freudenberg (ed.), World
Scientific, QP--PQ Series 15 (2003) 151--159

\bibitem{[Miyad05]}
Miyadera T.:
 Entangled Markov Chains generated by symmetric channels,
\textit{Inf. Dim. Analysis, Quantum Probab. Related Topics} {\bf 8}
(2005) 497--504.

\bibitem{[Ohno03]}
Ohno H.: Extendability of generalized quantum Markov chains on gauge
invariant $C^*$--algebras, \textit{Inf. Dim. Analysis, Quantum
Probab. Related Topics},  {\bf 8}(2005) 141--152.

\bibitem{[Pr]} Preston C.:
Gibbs States on Countable Sets (Cambridge University Press, London,
1974).

\bibitem{Spa} Spataru A.: Construction of a
Markov Field on an infinite Tree, \emph{Advances in Math} {\bf
81}(1990), 105--116.

\bibitem{[Sp75]}
Spitzer F.: Markov random Fields on an infinite tree, \emph{Ann.
Prob.} {\bf 3} (1975) 387-398.

\bibitem{[Za83]}
Zachary S.: Countable state space Markov random fields and Markov
chains on trees, \emph{Ann. Prob.} {\bf 11} (1983) 894--903.

\bibitem{[Za85]}
Zachary S.: Bounded attractive and repulsive Markov specifications
on trees and on the one-dimensional lattice, \emph{Stochastic
Process. Appl.} {\bf 20} (1985) 247--256.


\end{thebibliography}
\end{document}